\documentclass[11pt,psamsfonts]{amsart}
\usepackage{amsmath}
\usepackage{amsthm}
\usepackage{amssymb}
\usepackage{amscd}
\usepackage{amsfonts}
\usepackage{amsbsy}
\usepackage{graphicx}
\usepackage[active]{srcltx}

\def\be{\begin{equation}}
\def\ee{\end{equation}}
\def\bq{\begin{eqnarray}}
\def\eq{\end{eqnarray}}

\def\beq{\begin{eqnarray*}}
\def\eeq{\end{eqnarray*}}

\newtheorem {theorem} {Theorem}
\newtheorem {proposition} [theorem]{Proposition}

\newtheorem {corollary} [theorem]{Corollary}
\newtheorem {lemma}  [theorem]{Lemma}

\newcommand{\R}{\mathbb{R}}

\newcommand{\Q}{\mathbb{Q}}
\newcommand{\N}{\mathbb{N}}

\begin{document}

\title[Existence of  numerically detected limit cycles]
{Proving the existence of  numerically detected planar limit cycles}

\author[A. Gasull, H. Giacomini and M. Grau]
{Armengol Gasull$^1$, H\'ector Giacomini$^2$ and Maite Grau$^3$}

\address{$^1$ Departament de Matem\`{a}tiques, Universitat
Aut\`{o}noma de Barcelona, 08193 Bellaterra, Barcelona, Catalonia,
Spain}

\email{gasull@mat.uab.cat}

\address{$^2$ Laboratoire de Math\'ematiques et Physique Th\'eorique.
C.N.R.S. UMR 7350., Facult\'e des Sciences et Techniques.
Universit\'e de Tours., Parc de Grandmont 37200 Tours, France.}

\email{Hector.Giacomini@lmpt.univ-tours.fr}

\address{$^3$ Departament de Matem\`atica, Universitat de Lleida,
Avda. Jaume II, 69; 25001 Lleida, Catalonia, Spain}

\email{mtgrau@matematica.udl.cat}


\date{}
\dedicatory{}
\thanks{Corresponding author: Maite Grau. E-mail: \texttt{mtgrau@matematica.udl.cat}}

\maketitle

\begin{abstract}
This paper deals with the problem of location and existence of limit cycles for real planar
polynomial differential systems. We provide a method to construct Poincar\'e--Bendixson regions by
using transversal curves, that enables us to prove the existence of a limit cycle that has been
numerically detected. We apply our results to several known systems, like  the Brusselator one
or some Li\'{e}nard systems,  to prove the existence of the limit cycles and to locate them very
precisely  in the phase space. Our method, combined with some other classical tools can be applied
to obtain  sharp bounds for the bifurcation values of a saddle-node bifurcation of limit cycles, as we do for the Rychkov system.
\end{abstract}

\section{Introduction\label{sect1}}

We consider real planar polynomial differential systems of the form
\begin{equation} \label{eq1} \dot{x}   =  dx/dt=  P(x,y), \quad
\dot{y}   = dy/dt=   Q(x,y), \end{equation} where $P(x,y)$ and
$Q(x,y)$ are real polynomials. We denote by $X=(P,Q)$ the vector
field associated to \eqref{eq1} and $z=(x,y)$. So, \eqref{eq1} can
be written  as $\dot z=X(z).$

When dealing with system (\ref{eq1}) one of the main problems is to determine the number and location
of its limit cycles. Recall that a limit cycle is an isolated periodic orbit of the system.
For a given vector
field, when it is not very near of a bifurcation, the limit cycles can usually be detected by
numerical methods. A bifurcation is a qualitative change in the behaviour of a vector field as a
parameter of the system is varied. This phenomenon can involve a change in the stability of a
limit cycle or the creation or destruction of one or more limit cycles. If a periodic orbit is stable (unstable),
then forward (backward) numerical integration of a trajectory with an initial condition in its
basin of attraction will converge to the periodic orbit as $t\to\infty$ ($t\to-\infty$). Once for a given vector field a
limit cycle is numerically detected  there is no general method to rigourously prove  its existence. In this work  we present a procedure to prove the existence  of a limit cycle in that situation. The method
is based on the Poincar\'e--Bendixson theorem, see for instance \cite{DuLlAr, Perko} and also
Theorem~\ref{thpb}. Poincar\'e--Bendixson theorem (cf. Theorem~\ref{thpb}) can be very useful to prove the existence of a limit cycle and to give a region where it is located. However, this result is hardly found in applications due to the
difficulty of constructing the boundaries of a Poincar\'e--Bendixson region. Our aim in this work
is to give a constructive procedure for finding transversal curves which
define Poincar\'e--Bendixson regions and thus, to prove the
existence of limit cycles that have been numerically detected.

\smallskip
Consider a smooth  and non-empty curve $C$ in $\mathbb{R}^2$. Let $C   =   \{  z ( s )  = ( x(
s ),  y( s ))\,  :\,
 s  \in \mathcal{I} \}$ be a class $\mathcal{C}^1$ parametrization of  $C$, where
 $\mathcal{I}$ is a real interval. It is said that $C$ is
{\it regular} if $ z '( s ) \neq (0,0)$ for all $ s  \in
\mathcal{I}$. Given  $z=(x,y)$ we set $z^{\perp}=(y,-x)$ and
$(x_1,y_1)\cdot(x_2,y_2)=x_1 x_2+y_1 y_2$. A \emph{contact point} with the flow given by~\eqref{eq1} is a point $z(s)$ such that the tangent vector to $C$ at this point, $z'(s)$  is parallel to $X(z(s)).$

As usual, we will say that a curve $C$ is \emph{transversal} with respect to the flow given
by~\eqref{eq1}
if the scalar product \[X(z(s))\cdot ( z'(s))^{\perp} =  P\left( z ( s )\right) y'( s ) -
Q\left( z ( s )\right)    x'( s )\]
does not change sign and vanishes only on
finitely many contact points. When the above scalar product does not vanish we will say that the
curve is {\it strictly transversal}. Notice that intuitively, these definitions mean that the flow
of system {\rm (\ref{eq1})} ``crosses $C$ in the same direction" on all its points.

 A $\mathcal{C}^1$ closed plane curve $C$ is a
regular parameterized curve $ z : [a,b] \longrightarrow \mathbb{R}^2$ such that $ z $ and  its
derivative coincide at $a$ and $b$. The curve is said to be simple if it has no
self-intersections, that is if $ s _1,  s _2 \in [a,b)$ and $ s _1 \neq  s _2$, then $ z ( s _1)
\neq  z ( s _2)$. For further information about these classical concepts, see for instance
\cite{doCarmo}.

\smallskip

 A transversal section of system (\ref{eq1}) is an arc of a curve without
 contact points. Given a limit cycle $\Gamma$ there always exist a transversal
 section $\Sigma$ which can be parameterized by $r \in (-\rho,\rho)$
 with $\rho>0$ and $r=0$ corresponding to a common point between $\Gamma$
 and $\Sigma$. Given $r \in (-\rho,\rho)$, we consider the flow of
 system (\ref{eq1}) with initial point the one corresponding to $r$ and we follow
 this flow for positive values of $t$. It can be shown, see for instance \cite{Perko},
 that for $\rho$ small enough, the flow cuts $\Sigma$ again at some point corresponding to the parameter
 $\mathcal{P}(r)$. The map $r \longrightarrow \mathcal{P}(r)$ is called the
 Poincar\'e map associated to the limit cycle $\Gamma$ of system (\ref{eq1}).
 It is clear that $\mathcal{P}(0)   =   0$. If $\mathcal{P}'(0) \neq 1$, the
 limit cycle $\Gamma$ is said to be {\it hyperbolic}. If the expansion of
 $\mathcal{P}(r)$ around $r=0$ is of the form $\mathcal{P}(r)   =
   r   +   a_\mu r^\mu   +   \mathcal{O}(r^{\mu+1})$ with $a_{\mu}
 \neq 0$ and $\mu \geq 2$, we say that $\Gamma$ is a multiple limit cycle
 of multiplicity $\mu$. A classical result, see for instance \cite{Perko},
 states that if $\Gamma   =   \left\{ \gamma(t)   :   t \in [0,T) \right\}$,
  where $\gamma(t)$ is the parametrization of the limit cycle in the time variable
   $t$ of system (\ref{eq1}) and $T>0$ is the period of $\Gamma$, that is, the
   lowest positive value for which $\gamma(0) = \gamma(T)$, and $\gamma(0) =
   \Gamma \cap \Sigma$, then
\begin{equation*} \mathcal{P}'(0)   =   \exp \left\{ \int_{0}^{T}
\operatorname{div} X\left( \gamma(t) \right) dt \right\}, \end{equation*} where
\[\operatorname{div} X(x,y)
= \frac{\partial P}{\partial x}(x,y)   +   \frac{\partial Q}{\partial y}(x,y) \]
is the {\em
divergence} of $X$. Hence
\[
k:= \int_{0}^{T}
\operatorname{div} X\left( \gamma(t) \right) dt\ne0
\]
is the condition for a limit cycle $\Gamma$ to be hyperbolic. It is
clear that if $k>0$ (resp. $k<0$), then $\Gamma$ is an unstable
(resp. stable) limit cycle. If $\Gamma$ is a multiple limit cycle of
multiplicity $\mu$ and $\mu$ is odd, then $\Gamma$ is unstable if
$a_\mu>0$ and stable if $a_\mu<0$. If $\mu$ is even, then the limit
cycle $\Gamma$ is said to be semi-stable. For the definitions and
related results, see for instance \cite{DuLlAr, Perko, YanQian}.

The Poincar\'e--Bendixson theorem, which can be found for instance in  \cite[Sec. 1.7]{DuLlAr} or
in \cite[Sec. 3.7]{Perko}, has as a corollary the following result which motivates the definition
of Poincar\'e--Bendixson region. See also Theorem 4.7 of \cite[Chap. 1]{ZhangZF}.
\begin{theorem} \label{thpb} {\rm [Poincar\'{e}-Bendixson annular Criterion]}
Suppose that $R$ is a finite region of the plane $\mathbb{R}^2$ lying between two
$\mathcal{C}^1$ simple disjoint
closed curves $C_1$ and $C_2$. If
\begin{itemize}
\item[{\rm (i)}] the curves $C_1$ and $C_2$ are transversal for system
{\rm (\ref{eq1})} and the flow crosses them towards the interior of $R$, and
\item[{\rm (ii)}] $R$ contains no critical points.
\end{itemize}
Then, system {\rm (\ref{eq1})} has an odd number of limit cycles
(counted with multiplicity) lying inside $R$.
\end{theorem}

In such a case, we say that $R$ is a \emph{Poincar\'e--Bendixson annular region}
for system~(\ref{eq1}).

\smallskip

As we have already stated our aim is to find transversal curves
which define Poincar\'e--Bendixson annular regions and thus, to
prove the existence of limit cycles, as well as to locate them. In
the paper \cite{GGconic} we dealt with the same problem and we
described a way to provide transversal conics which give rise to a
Poincar\'e--Bendixson annular region. In this previous paper we
treated several examples for which we numerically knew the existence
of a limit cycle, but we did not use this information. Besides, we
could not ensure the existence of the transversal conics. In the
present work we give an answer to the following question: if one
numerically knows the existence of a hyperbolic limit cycle, can one
analytically prove the existence of such limit cycle? In section
\ref{sect2} we describe a method which answers this question in an
affirmative way.

\smallskip

The following theorem is the main result of this paper and it gives
the theoretical basis of the method described in section
\ref{sect2}. We prove:

\begin{theorem} \label{th1}
Let $\Gamma=\{ (\gamma(t)\, :\,   t \in [0,T] \}$ be a $T$-periodic hyperbolic limit cycle
of~\eqref{eq1}, parameterized by the time $t$. Define
\[
\tilde z_\varepsilon(t)= \gamma(t)+\varepsilon \tilde{u}(t) (\gamma'(t))^\perp,\] where
\begin{equation} \label{defu} \tilde{u}(t)   =   \frac{1}{||\gamma'(t)||^2}
\exp\left\{ \int_{0}^{t}\operatorname{div} X\left(\gamma(s)\right) \,ds   -   \kappa\, t \right\}
\end{equation} and
 $ \kappa   = \frac kT=
\frac{1}{T} \int_{0}^{T}\operatorname{div}X(\gamma(t))\, dt.$ Then,  the curve
$\left\{\tilde{z}_\varepsilon(t)\, :\, t \in [0,T] \right\}$ is $T$-periodic and, for
$|\varepsilon|>0$ small enough, it is strictly transversal to the flow associated to system
\eqref{eq1}.
\end{theorem}
The proof of this result is given in section \ref{sectp}. Note that
in its statement $\tilde{u}(t)>0$ for all $t$ and  $\kappa \neq 0$
because $\Gamma$ is hyperbolic.

Notice that as a consequence of the above result, the curve $\tilde{z}_\varepsilon(t)$ is a
transversal oval close to the limit cycle $\Gamma$ for $|\varepsilon|>0$ small enough, which is
inside or outside it depending on the sign of $\varepsilon$.

As an illustration of the effectiveness of our approach we apply it to
locate the limit cycles in two celebrated planar differential
systems, the van der Pol oscillator and the Brusselator system, see
sections~\ref{vdp} and~\ref{bru}, respectively. As we will see, the
van der Pol limit cycle  is ``easier" to be treated than the one of
the Brusselator system. In section~\ref{nova} we give an
explanation for the different level of difficulty for studying both limit
cycles. We prove that the different level of difficulty is hidden in the
sizes of the respective Fourier coefficients of the two limit
cycles, see Theorem~\ref{fita}. This theorem also shows that our
approach for detecting strictly transversal closed curves always works in
finitely many steps.

Finally, to show the applicability of the method to detect
bifurcation values, we use it  to find a sharp interval for  the
bifurcation value for a saddle-node bifurcation of limit cycles for
the Rychkov system. Recall that a saddle-node  bifurcation of limit cycles occurs when a stable limit cycle and an unstable limit cycle coalesce and become a double semi-stable limit cycle. A  saddle-node  bifurcation of limit cycles corresponds to an elementary catastrophe of fold type.

In 1975 Rychkov(\cite{Rychkov75}) proved that the system
\[ \dot{x} \, = \, y-\left( x^5-\mu x^3 + \delta x \right),
\quad \dot{y} \, = \, -x, \] with $\delta, \mu \in \mathbb{R},$ has
at most 2 limit cycles. Moreover, it is known that it has 2 limit
cycles if and only if $\delta>0$ and $0<\delta<\Delta (\mu),$ for
some unknown function $\Delta .$ For the value $\delta=\Delta (\mu)$
the system has a double limit cycle and, varying $\delta$, it
presents a saddle-node bifurcation of limit cycles. This system is
also studied by Alsholm(\cite{Alsholm92}) and Odani(\cite{Odani96}).
In particular Odani proved that $\Delta (\mu)>{\mu^2}/5.$

We believe that it is an interesting challenge to develop methods
for finding sharp estimations of $\Delta (\mu)$. Here we will fix
our attention on $\delta^*:=\Delta (1).$ Notice that Odani's result
implies that $\delta^*>1/5=0.2.$  We prove:

\begin{theorem}\label{rych} Let $\delta=\delta^*$ be the value for which the Rychkov system
\begin{equation}\label{rych1} \dot{x} \, = \, y-\left( x^5- x^3 + \delta x \right), \quad
\dot{y} \, = \, -x \end{equation}  has a semi-stable limit cycle.
Then $0.224\,{<}\,\delta^*\,{<}\,0.2249654.$
\end{theorem}

The lower bound for $\delta^*$  is proved by using the tools
introduced in this work. The upper bound is proved by constructing a
polynomial function in $(x,y)$ of very high degree such that its total
derivative with respect to the vector field does not change sign.
This method is proposed and already developed for general classical
Li\'{e}nard systems by Cherkas(\cite{Cher}) and also by
Giacomini-Neukirch (\cite{GiaNeu97,GiaNeu98}).

\section{Proof of Theorem \ref{th1} and a corollary\label{sectp}}\label{sec2}

\begin{proof}[Proof of Theorem {\rm \ref{th1}}]
To prove that the curve $\tilde{z}_\varepsilon(t)$ is  $T$-periodic simply notice that $\gamma(t)$
is  $T$-periodic and that the function $\tilde{u}(t)$ is $T$-periodic as well, due to its
definition (\ref{defu}), because for any real integrable $T$-periodic function $h$, the new
function
\[
H(t)= \int_0^t h(s)\,ds- \frac t T\int_0^T h(s)\,ds
\]
is also $T$-periodic.

Now, we show that the curve $\tilde{z}_\varepsilon(t)$,  for $|\varepsilon|>0$ small enough,  is
strictly transversal to system \eqref{eq1}. This follows once we prove  that
\begin{equation}\label{objectiu}
X(\tilde z_\varepsilon(t))\cdot ( \tilde z_\varepsilon'(t))^{\perp} =\kappa \mathcal{E}(t)
\,\varepsilon +  \mathcal{O}(\varepsilon^2),
\end{equation}
where we have introduced, to simplify notation,
\[ \mathcal{E}(t)   :=   \exp\left\{ \int_{0}^{t}
\operatorname{div} X\left(\gamma(s)\right) \,ds   -   \kappa\, t \right\}>0. \]

Let us prove~\eqref{objectiu}. We drop the dependence on $t$ to simplify notation. Since $\tilde
z_\varepsilon=\gamma+\varepsilon \tilde u \gamma'^\perp,$ we have that $\tilde
z_\varepsilon^\perp=\gamma^\perp-\varepsilon \tilde u \gamma'$. Then
\begin{align*}
X(\tilde z_\varepsilon)&= X(\gamma+\varepsilon \tilde u \gamma'^\perp)=X(\gamma)+ \varepsilon
\tilde u DX(\gamma)\gamma'^\perp+\mathcal{O}(\varepsilon^2)\\ &=\gamma'+ \varepsilon \tilde u
DX(\gamma)\gamma'^\perp+\mathcal{O}(\varepsilon^2),\\
\tilde z_\varepsilon'^\perp &=\gamma'^\perp-\varepsilon \tilde u' \gamma'-\varepsilon \tilde
u\gamma''.
\end{align*}
Hence $X(\tilde z_\varepsilon)\cdot \tilde z_\varepsilon'^\perp =
\gamma'\cdot\gamma'^\perp+\tau\varepsilon  +\mathcal{O}(\varepsilon^2)= \tau\varepsilon
+\mathcal{O}(\varepsilon^2),$ where
\[\tau=-\tilde u' \gamma'\cdot \gamma' - \tilde u\gamma'\cdot
\gamma''+ \tilde u DX(\gamma)\gamma'^\perp\cdot \gamma'^\perp.\]
 To simplify $\tau$, note that
\begin{align*}
\mathcal{E}' &=\left( \operatorname{div} X\left(\gamma\right)-\kappa \right)\mathcal{E},\quad
\tilde{u}   =   \frac{\mathcal{E}}{||\gamma'||^2}= \frac{\mathcal{E}}{\gamma'\cdot\gamma'},\\
 \tilde{u}'  &=\left( \operatorname{div} X\left(\gamma\right)-\kappa
-2\frac{\gamma'\cdot\gamma''}{\gamma'\cdot\gamma'} \right)\tilde{u}.
\end{align*}
Therefore,
\begin{align*}
\tau=& \Big( \big(\kappa- \operatorname{div}
X\left(\gamma\right)\big)\gamma'\cdot\gamma'+2\gamma'\cdot\gamma'' -\gamma'\cdot \gamma''+
DX(\gamma)\gamma'^\perp\cdot \gamma'^\perp\Big)\tilde u.
\end{align*}
Using that $\gamma'=X(\gamma)$ we have that $\gamma''=DX(\gamma)\gamma'.$ Thus,
\begin{align*}\tau=& \Big( \big(\kappa- \operatorname{div}
X\left(\gamma\right)\big)\gamma'\cdot\gamma'+ DX(\gamma)\gamma'\cdot \gamma'+
DX(\gamma)\gamma'^\perp\cdot \gamma'^\perp\Big)\tilde u.
\end{align*}
Finally, we use the following simple and nice formula
\[
\big(A v\big) \cdot v+ \big(A v^\perp)\cdot v^\perp= \operatorname{trace}(A)\,v\cdot v,
\]
where $A$ is a $2\times2$ matrix and $v$ is a vector. Hence,
\begin{align*}\tau=& \Big( \big(\kappa- \operatorname{div}
X\left(\gamma\right)\big)\gamma'\cdot\gamma'+ \operatorname{div} X(\gamma) \gamma'\cdot\gamma'
\Big)\tilde u= \kappa\,\tilde u \gamma'\cdot\gamma'=\kappa\, \mathcal{E},
\end{align*}
as we wanted to prove.
\end{proof}

\smallskip

In fact, in the above proof,  to show the transversality it is not used the specific value of
$\kappa$. We only have used that it is a nonzero constant. Hence, the following result holds:

\begin{corollary}
Given an orbit $\left\{ \gamma(t)  :  t \in (0,t_1) \right\}$ of system {\rm (1)}, parameterized
by the time $t$, and a nonzero constant $K$, then for $|\varepsilon|>0$ small enough,  the curve
\[
\hat z_{K,\varepsilon}(t)\, = \, \gamma(t) \, + \, \varepsilon \hat{u}_K(t) (\gamma'(t))^\perp,
\]
 where
\begin{equation*}  \hat{u}_K(t)   =   \frac{1}{||\gamma'(t)||^2}
\exp\left\{ \int_{0}^{s}\operatorname{div} X\left(\gamma(s)\right) \,ds   -   K t,\right\}
\end{equation*}
 is strictly transversal to the flow given by system \eqref{eq1}.
\end{corollary}

The proof goes as in the proof of Theorem \ref{th1}, showing that
\[
X(\hat u_{K,\varepsilon}(t))\cdot ( \hat u_{K,\varepsilon}'(t))^{\perp}= K \mathcal{E}_K(t)
\varepsilon + \mathcal{O}(\varepsilon^2),
\]
where
\[ \mathcal{E}_K(t)   :=   \exp\left\{ \int_{0}^{t}
\operatorname{div} X\left(\gamma(s)\right) \,ds   -   K t \right\}>0. \] Hence, for
$|\varepsilon|>0$ small enough, the sign of $K\,\varepsilon$ determines  how the flow of system
\eqref{eq1} crosses the piece of curve $\hat z_{K,\varepsilon}.$

This corollary can be useful to construct curves without contact to a piece $\Gamma$ of solution
of \eqref{eq1},
not closed, which are ``parallel" to it and such the flow crosses them either
towards $\Gamma$ or in the opposite direction, as desired.

\section{Description of the method\label{sect2}}

\subsection{First step: the ``numerical'' limit cycle}

Assume that system (\ref{eq1}) has a hyperbolic limit cycle $\Gamma$. To simplify the notation,
we also assume that  the segment $\Sigma   := \{
(x_0,0)   :   \alpha < x_0 < \beta \}$ is a transversal section, $0 \leq \alpha < \beta$. Given a
point $(x_0,0) \in \Sigma$, we can numerically compute the
solution $\varphi(t;x_0)$ with initial condition $\varphi(0;x_0)   =   (x_0,0).$ We denote the scalar components of the function $\varphi \, = \, (\varphi_1, \varphi_2)$.

 We can also numerically compute the value $T(x_0)>0$ for
which
\[ \varphi(T(x_0);x_0) \in \Sigma \] and $T(x_0)$ is the lowest one with
this property.

We look for a zero of the {\em displacement map}
\[ \mathcal{P}(x_0)-x_0=\varphi_1(T(x_0);x_0)    -   x_0, \] and we can find the zero $x_0^*$
of this map with as much precision as the computer allows. In this way,
we have {\em numerically} computed the limit cycle $\left\{ \varphi(t;x_0^*)
   :   t \in [0,T(x_0^*)] \right\}$ and its period $T(x_0^{*})$.

{F}rom now on, even though we do not have analytic but numerical
expressions, we denote the limit cycle by $\gamma(t)$ and its period
by $T$.

\subsection{Second step:  the {\em numerical} transversal curve}

We can numerically compute
\[ \kappa   =   \frac{1}{T} \int_{0}^{T}\operatorname{div} X \left(\gamma(t)\right) \,dt \]
and a tabulation of the function $\tilde{u}(t)$ given in~\eqref{defu}. As we will see,
we even do not need to care about
the method used to get this approximation (for instance it can be spline interpolation)
because from a point on,
our method starts again and only does analytic computations.

\smallskip

Next, we fix a value of  $\varepsilon$ and we construct
\[
\tilde z_\varepsilon(t) \, =\,  \gamma(t)\, + \, \varepsilon \tilde u (t) \gamma'(t)^\perp.
\]
 We
numerically check whether the above curve
is transversal to system {\rm (\ref{eq1})}. If not, we take a smaller
value of $|\varepsilon|$.

\smallskip

We take an odd natural number $n$ and, from these computations, we get a list of $n$
points of the curve $\tilde z_\varepsilon(t)$. For instance these points are the ones
corresponding
to times $t=i/T, i=0,1,\ldots n-1.$

\subsection{Third step: a first explicit transversal  curve}

{F}rom the previous step we have a list of $n$ points of the curve
$\tilde{z}_\varepsilon(t)$. Since we have chosen $n$ odd, we define
$m=(n-1)/2$. We consider the expressions
\begin{equation}\label{rat} \begin{array}{lll} \tilde{w}_1^{(m)}(\theta) & = & \displaystyle
\tilde{c}_{0,0}
   +   \sum_{i=1}^{m} \tilde{c}_{i,0} \cos(i \theta) + \tilde{c}_{0,i}
 \sin(i \theta), \vspace{0.2cm} \\
\tilde{w}_2^{(m)}(\theta) & = & \displaystyle \tilde{d}_{0,0}   + \sum_{i=1}^{m} \tilde{d}_{i,0}
\cos(i \theta) + \tilde{d}_{0,i} \sin(i \theta), \end{array} \end{equation} with undefined
coefficients $\tilde{c}_{i,j}, \tilde{d}_{i,j}$. We have $2(2m+1)$ unknowns.

\smallskip

We impose that the curve $\tilde{w}^{(m)}(\theta)=(\tilde w_1^{(m)}(\theta),\tilde w_2^{(m)}(\theta))$
 passes through the list of $n$ points when $\theta   =   2
 \pi   i /n$ for $i=0,1,2,\ldots, n-1$. We also have $2n=2(2m+1)$ conditions.

\smallskip

We obtain a curve $\left\{\tilde{w}^{(m)}(\theta)   :   \theta \in [0,2 \pi] \right\},$ which
approximates the numerical transversal curve $\left\{\tilde{z}_\varepsilon(t)  :   t \in [0,T]
\right\}.$

\subsection{Fourth step: a curve with rational coefficients}

We take rational approximations of the coefficients in the
expressions of $\tilde{w}^{(m)}(\theta)$. These rational approximations are taken with a certain precision. In case this precision is not sharp enough, it can be sharpened after the fifth step. We obtain a new closed
curve $\left\{ w^{(m)}(\theta)   :   \theta \in [0,2\pi] \right\},$
whose coefficients are rational. That is,
\begin{equation}\label{curve}
\begin{array}{lll} w_1^{(m)}(\theta) & = & \displaystyle c_{0,0}   +
\sum_{i=1}^{m} c_{i,0} \cos(i \theta) + c_{0,i} \sin(i \theta),
\vspace{0.2cm} \\
w_2^{(m)}(\theta) & = & \displaystyle d_{0,0}   +   \sum_{i=1}^{m}
d_{i,0} \cos(i \theta) + d_{0,i} \sin(i \theta), \end{array}
\end{equation} where $c_{i,j}$ and $d_{i,j}$ are rational numbers which are
approximations of the corresponding $\tilde{c}_{i,j}$ and
$\tilde{d}_{i,j}$.

\subsection{Fifth step: a transversal curve}

We have constructed a closed curve $\{ w^{(m)}(\theta)  :   \theta \in [0,2\pi]\},$ whose
coefficients are rational. We know that this curve is transversal to system {\rm (\ref{eq1})} if
\[
f(\theta)= X(w^{(m)}(\theta))\cdot (w^{(m)})'(\theta)^\perp
\]
does not change sign for all $\theta \in [0,2\pi]$. To prove so, we expand $f(\theta)$ in powers
of $\cos \theta, \sin \theta$ and we change $\cos \theta$ by $u$ and $\sin \theta$ by $v$. Then we
take the resultant of this expression with $u^2+v^2-1$ with respect to $v$, see for
instance~\cite{st}. If this resultant, $R(u)$, which is a polynomial in $u$, has no real  roots
for $u \in [-1,1]$  we know that the first polynomial has no common real solutions with
$u^2+v^2=1$ and as a consequence $f(\theta)$ does not vanish. To prove that $R(u)$ has no real roots in $[-1,1]$ one can compute
for instance the Sturm sequence of $R$ and apply the Sturm theorem, see for instance \cite{sb}.

Recall that we have taken the coefficients of $w(\theta)$ rational. We remark  that if all  the coefficients of the vector field $X(x,y)$ which define the system {\rm (\ref{eq1})} are also rational, the computations needed to ensure that $f(\theta)$ does not change sign are much simpler.

\smallskip

If the obtained curve is not transversal, we take the rational approximations of the fourth step with a higher precision. Another option is to repeat from the third step in order to obtain a list
of $n=2m+1$ points of the curve $\tilde{w}^{(m)}(t)$ with a higher $n$.

\subsection{Sixth step: a Poincar\'e--Bendixson annular region}

We repeat the above five steps process with an $\varepsilon$ of
different sign in order to obtain an inner transversal curve and an
outer transversal curve to the limit cycle. In this way, we have a
Poincar\'e--Bendixson annular region which analytically shows the
existence of at least one limit cycle in its interior. We can take
smaller values of $\varepsilon$ which will make this region
narrower. Thus, we locate the limit cycle.

\section{Examples\label{sect3}}

We present a couple of examples for which a Poincar\'e--Bendixson region can be
constructed by using the method described above.

\subsection{Example 1: the van der Pol system\label{sect31}}\label{vdp}

We start with the celebrated van der Pol system
\begin{equation} \label{eqex1} \dot{x}   =   y-\varepsilon \left( \frac{x^3}{3}
-x \right), \quad \dot{y}   =   -x, \end{equation} with $\varepsilon >0$.

The origin is the only finite critical point of the system and it is
a repulsive point (a focus when $0<\varepsilon<2$ and a node when
$\varepsilon \geq 2$). It is known, see for instance \cite{Perko},
that system (\ref{eqex1}) has a unique stable and hyperbolic limit
cycle for all $\varepsilon>0$ which bifurcates from the circle of
radius $2$ when $\varepsilon = 0$ and which disappears into a
slow-fast periodic limit set when $\varepsilon \to +\infty$. The
semi-axis $\Sigma   := \{(x_0,0)
  :   x_0>0\}$ is a transversal section for the limit cycle.

\begin{figure}[htb]
\includegraphics[width=0.5\textwidth]{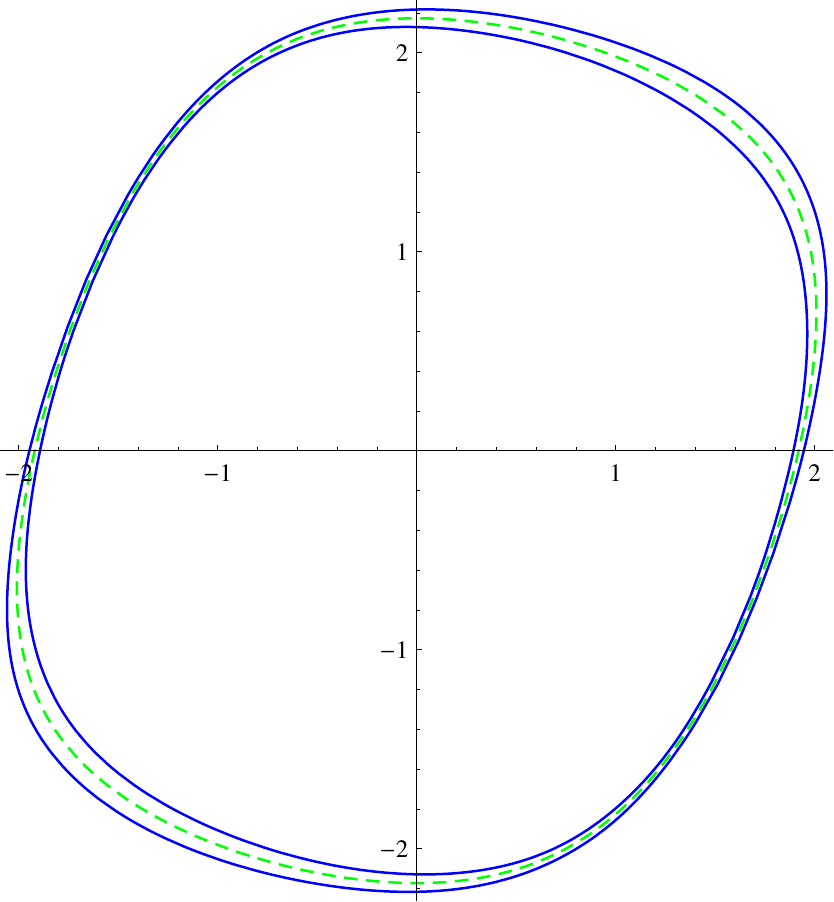}
\caption{The transversal curves are represented in blue and the (numerical) limit cycle in dashed green,
for the van der Pol system \eqref{eqex1} with $\varepsilon=1$.} \label{vdp-m12}
\end{figure}

We consider the van der Pol system with $\varepsilon   =   1.$
The limit cycle crosses the transversal section $\Sigma$ at $x_0^* \sim  1.91928$ and
it has period $T \sim 6.6632866$. We have numerically computed the limit cycle and from this approximation we have obtained the described values of $x_0^*$ and $T$.

By our method we obtain an inner transversal curve and an outer transversal curve
$ w_{\mbox{in}}(\theta)$ and $w_{\mbox{ex}}(\theta)$ with $\theta \in [0,2\pi]$,  which provide a
Poincar\'e--Bendixson
annular region. The inner transversal curve cuts $\Sigma$ at $\sim 1.89331$ and the
outer transversal curve at $\sim 1.94543$, see Figure \ref{vdp-m12}.

The inner transversal curve is obtained with $\varepsilon   =   0.05$ and $m=12$. By the numerical
computations, we obtain the following list of $n=2m+1  =   25$ points:
\[ \begin{array}{l}
\displaystyle \big\{ (1.89451,  0.0056435),   (1.76278, -0.488101),   (1.59066, -0.939363),
\\ \displaystyle (1.38198, -1.33813),  (1.12999, -1.67325),  (0.819552, -1.92987),
\\ \displaystyle (0.424859, -2.08912),  (-0.093381, -2.12507),  (-0.747354, -2.00013),
\\ \displaystyle (-1.39679, -1.69586),  (-1.80051, -1.27605),   (-1.9537, -0.788939),
\\ \displaystyle (-1.93903, -0.264387),   (-1.83453, 0.245845),   (-1.68122, 0.719683),
\\ \displaystyle (-1.49111, 1.14594),  (-1.26215, 1.51447),  (-0.98337, 1.81246),
\\ \displaystyle (-0.634949, 2.02302),  (-0.183705, 2.12466),   (0.40691, 2.08508),
\\ \displaystyle (1.08983, 1.86873),  (1.63579, 1.49426),   (1.90318, 1.04111),
\\ \displaystyle (1.96187,  0.527263) \big\}. \end{array} \]

These points are represented in Figure \ref{vdp-inner-curpoints}.
Applying our method, we find a curve of the form \eqref{rat} with
$m=12,$ which passes through these points, see Figure
\ref{vdp-inner-curpoints}.

\begin{figure}[htb]
\includegraphics[width=0.6\textwidth]{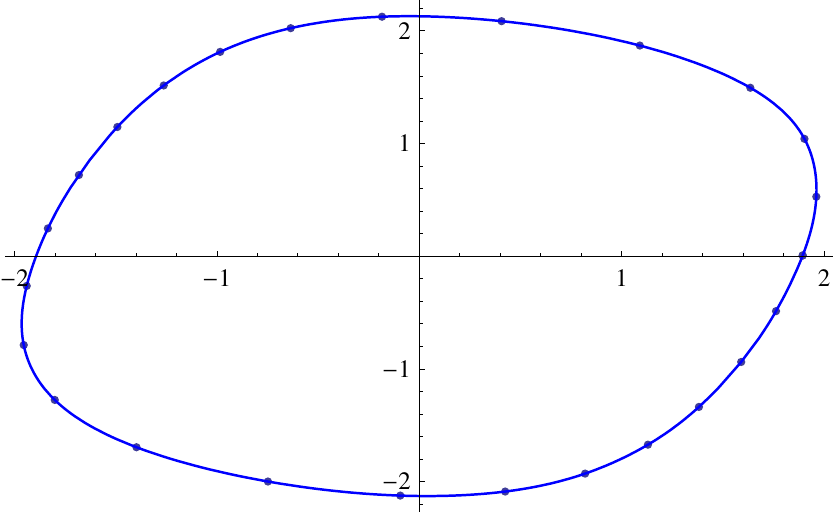}
\caption{Points of the transversal curve to the limit cycle and the approximated transversal curve
to the limit cycle of the van der Pol equation with $\varepsilon=1$, numerically computed.}
\label{vdp-inner-curpoints}
\end{figure}

For an a priori chosen precision we replace the coefficients in the above curve by rational
numbers. In this particular case we obtain:

\begin{align*} \displaystyle w_{1,\mbox{in}}(\theta)   = &  \frac{1}{213892} +
\frac{18566}{9395} \cos
   (\theta)+\frac{\cos (2 \theta)}{117817}-\frac{1973}{35647} \cos
   (3 \theta)
   \vspace{0.2cm} \\ \displaystyle
   &-\frac{\cos(4\theta)}{84836}-\frac{337}{9801}\cos (5\theta)
   -\frac{\cos (6\theta)}{19746}+\frac{53 }{5756}\cos (7\theta)
   \vspace{0.2cm} \\ \displaystyle
   & -\frac{\cos (8\theta)}{420042} - \frac{\cos (9\theta)}{4738}
   + \frac{3 \cos (10\theta)}{11954} - \frac{\cos(11 \theta)
   }{776}  + \frac{\cos (12\theta)}{5488}
   \vspace{0.2cm} \\ \displaystyle
   & +\frac{1097 }{13625}\sin(\theta) - \frac{\sin (2\theta)}{103485}
   - \frac{2003}{9487}\sin(3 \theta) - \frac{\sin (4\theta)}{46332}
   \vspace{0.2cm} \\ \displaystyle
   & +\frac{1317}{54185}\sin (5 \theta) +\frac{\sin (6 \theta)}{85313}
   +\frac{103}{24125} \sin (7\theta) +\frac{\sin (8 \theta)}{8809}
   \vspace{0.2cm} \\ \displaystyle
   & -\frac{29 }{8781} \sin (9\theta) +\frac{\sin (10\theta)}{18036}
   +\frac{3 \sin (11\theta)}{7760} - \frac{7 \sin (12 \theta)}{12512}, \end{align*}

\begin{align*}
\displaystyle w_{2,\mbox{in}}(\theta)   = &  -\frac{1}{287689}+\frac{1207 \cos(\theta)}{18761}
 +\frac{\cos
(2\theta)}{180371}-\frac{721 \cos (3\theta)}{11644}
   \vspace{0.2cm} \\ \displaystyle
&+\frac{\cos (4\theta)}{46468}+\frac{116 \cos (5\theta)}{18697}
-\frac{\cos (6\theta)}{85239}-\frac{27 \cos (7\theta)}{11035}
   \vspace{0.2cm} \\ \displaystyle
&-\frac{\cos (8\theta)}{9627}-\frac{13 \cos (9\theta)}{9450}
-\frac{\cos (10\theta)}{19827}+\frac{7 \cos (11\theta)}{12142}
   \vspace{0.2cm} \\ \displaystyle
&+\frac{\cos (12\theta)}{2425}-\frac{22778 \sin (\theta)}{10867}
+\frac{\sin (2 \theta)}{106711}+\frac{295 \sin (3\theta)}{14827}
   \vspace{0.2cm} \\ \displaystyle
&-\frac{\sin (4\theta)}{98567}+\frac{35 \sin (5\theta)}{25042}
-\frac{\sin (6\theta)}{20630}-\frac{21 \sin (7\theta)}{7234}+\frac{\sin (8\theta)}{3180308}
   \vspace{0.2cm} \\ \displaystyle
&+\frac{21 \sin (9\theta)}{14432}+\frac{2 \sin (10
   \theta)}{9397}+\frac{7 \sin (11\theta)}{9435}
+\frac{\sin (12\theta)}{5087}.
 \end{align*}

We know that the curve is transversal to the system if the trigonometric polynomial
\[ f(\theta)   :=    X(w_{\mbox{in}}(\theta))\cdot w_{\mbox{in}}'(\theta)^\perp \] does not
change sign for all $\theta \in [0,2\pi]$. Since the polynomial $P(x,y)$ in the system is of
degree $3$ and the components of $w_{\mbox{in}}(\theta)$  are of degree $12$, we have that the
trigonometric polynomial $f(\theta)$ is of degree $48$. As we explained in the description of the
method, we expand $f(\theta)$ in powers of $\cos \theta, \sin \theta$ and we change $\cos \theta$
by $u$ and $\sin \theta$ by $v$, in order to get a polynomial $\tilde{f}(u,v)$ which is of degree
$48$. Then we take the resultant of $\tilde{f}(u,v)$ with $u^2+v^2-1$ with respect to $v$. This
resultant is a polynomial in $u$ of degree $96$. Finally we prove that this polynomial  has no
real roots for $u \in [-1,1]$ by computing its Sturm's sequence.

\smallskip

To obtain the outer transversal curve, we choose $\varepsilon   =   -0.05$ and $m=12$ and we
repeat the process. See Figure \ref{vdp-m12} for a representation of the inner and the outer
transversal curves together with the limit cycle.

\subsection{Example 2: the Brusselator system\label{sect32}}\label{bru}

We consider the system
\begin{equation} \label{eqex2} \dot{x}   =   a - (b+1) x + x^2y, \quad
\dot{y}   =   bx-x^2y, \end{equation} with $a,b>0$. This system has
a unique singular point at $(a,b/a)$. The semi-axis $\Sigma   :=
\left\{(x_0,b/a)   :   x_0>a \right\}$ is transversal to the flow.
If we take $a=1$ and $b=3$, the system exhibits a hyperbolic stable
limit cycle which cuts $\Sigma$ at $x_0^{*} \sim 2.30354344$ and has
period $T \sim 7.15691986$. We have numerically computed the limit cycle and the values of $x_0^*$ and $T$ have been obtained from this approximation. By our method we obtain an inner
transversal curve and an outer transversal curve, $w_{\mbox{in}}(\theta)$ and $w_{\mbox{ex}}(\theta)$ with $\theta \in
[0,2\pi]$, which provide a Poincar\'e--Bendixson annular region. The inner transversal curve cuts $\Sigma$ at $\sim 2.2981$ and the outer transversal curve at $\sim 2.3091$, see Figure \ref{Bruss-m140}.

\begin{figure}[htb]
\includegraphics[width=0.5\textwidth]{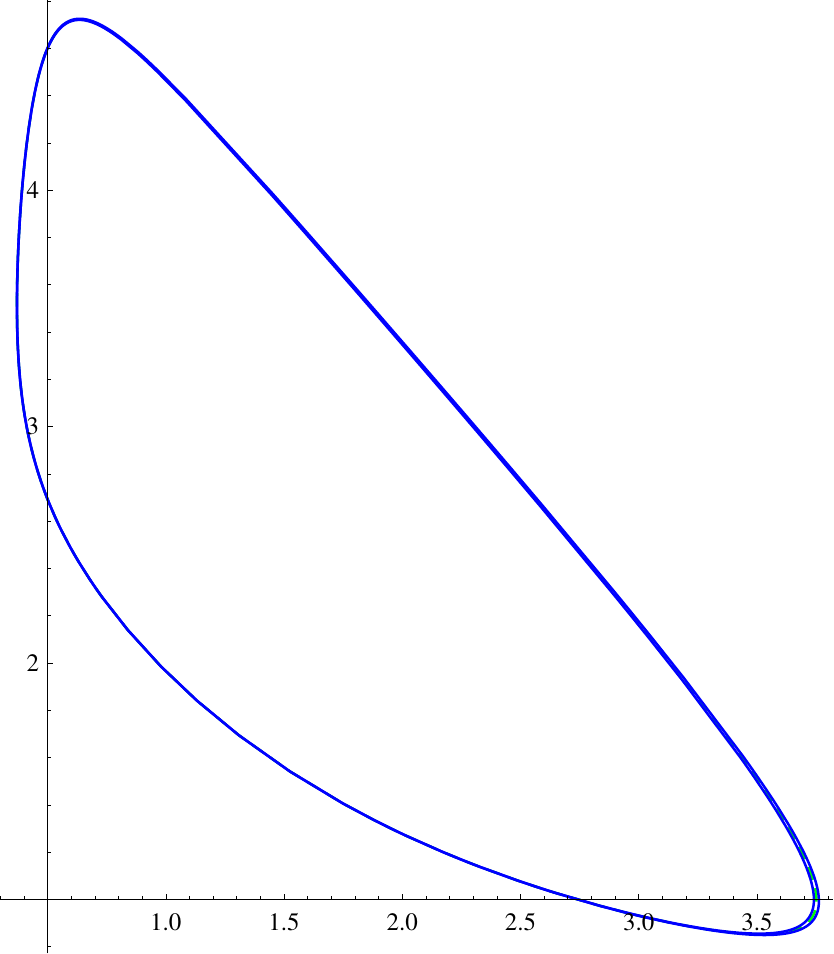}
\caption{The transversal curves are represented in blue and the (numerical) limit cycle in dashed green
for the Brusselator system with $a=1, b=3$. The three curves are almost indistinguishable.}
\label{Bruss-m140}
\end{figure}

The inner curve is obtained with $\varepsilon   =   0.05$ and the outer curve is obtained with
$\varepsilon   =   -0.05$, and both of them with $m=140$. We have not been able to find a
transversal curve with a lower value of $m$.

\smallskip

We also have considered system (\ref{eqex2}) with $a=1$ and when $b$
decreases. In this case the limit cycle shrinks until arriving to a weak
focus point when $b=2$ (Hopf bifurcation). We have studied the number of points ($2m+1$) needed to construct a transversal curve with our method giving an approximation of the limit cycle with similar accuracy. When $b=2.5$ with $\varepsilon=0.02$ we need to consider $m=55$ and when $b=2.2$ with $\varepsilon=0.007$ we need $m=30$.

\subsection{Comparison between the van der Pol and the Brusselator limit cycles}\label{nova}

Recall that by using our approach we  find closed transversal curves $\{w^{(m)}(\theta)\,:\,
\theta\in[0,2\pi]\}$ parameterized by the angle $\theta$ given by trigonometric polynomials of
degree $m$ with rational coefficients, see \eqref{rat}. In this section we convert these
curves into $T$-periodic ones simply by considering
\begin{equation}\label{www}
W^{(m)}(t)= w^{(m)}\left( \frac{2\pi}{T}t\right).
\end{equation}

As we have seen,  in the van der Pol system with $\varepsilon=1$, which is
\begin{equation} \label{vdpe1} \dot{x}   =   y-\left({x^3}/{3}-x\right),
\quad \dot{y}   =   -x, \end{equation} we can find a transversal
curve with $\varepsilon=0.05$ and $m=12$. On the other hand, for the
Brusselator system with $a=1$ and $b=3$, which is
\begin{equation} \label{Bruss13} \dot{x}   =   1-4x+x^2y, \quad
\dot{y}   =   3x-x^2y, \end{equation} we can find a transversal
curve taking $\varepsilon=0.05$ only with $m=140$ or higher. The aim
of this section is to understand why the number of points to be
taken, that is the value of $m$, is so different.

Before stating our main result we need to introduce some notations. If  $f$ is a $T$-periodic
continuous function,
\[ \|f\|_2   =   \sqrt{\frac{1}{T} \int_{0}^{T} f(s)^2   ds} \quad
\mbox{and} \quad \|f\|_{\infty}   =   \max\{|f(s)|  :   s \in [0,T]\} \] denote the $L_2$ and
$L_\infty$ norms, respectively. Notice  that $\|f\|_2 \leq \|f\|_\infty$. When $f$ is also a class
$\mathcal{C}^1$ function, its $\mathcal{C}^1$-norm is
\[ ||f||_{\mathcal{C}^1}= ||f||_\infty+||f'||_\infty. \]
 Similarly, for any of the three norms, when we
consider a $T$-periodic vector function $h(t)=(f(t),g(t))$,  we define $||h||=||f||+||g||.$

Finally, we denote by $\mathcal{F}_m (f)$ the Fourier polynomial of degree $m$ associated to $f$,
that is, \begin{equation}\label{trigo} \mathcal{F}_m (f)   =  \frac{ a_0}2   + \sum_{k=1}^{m}
a_k \cos\left(\frac{2\pi k}{T} t \right)   +    b_k \sin\left( \frac{2\pi k}{T} t
\right),\end{equation} where the constants $a_k,b_k, k=0,1,2,\ldots$
are
\begin{align*}
a_k=\frac 2T\int_0^T f(t) \cos\left(\frac{2\pi k }T t\right)\,dt,\quad b_k=\frac 2T\int_0^T f(t)
\sin\left(\frac{2\pi k }T t\right)\,dt.
\end{align*}
Similarly $\mathcal {F}_m(h)=(\mathcal{F}_m(f),\mathcal{F}_m(g)).$

We collect in the next proposition some well known results of
Fourier theory adapted to our interests, see for instance
\cite{Kor,Tol}. Some of the statements hold without our strong
hypotheses on $f.$

\begin{proposition}\label{propiet} Let $f$ be a $T$-periodic $\mathcal C^1$ function. The following holds:
\begin{enumerate}

\item[(i)] Let $p\ne \mathcal{F}_m(f)$ be any trigonometric polynomial of
degree $m$ (that is of the form~\eqref{trigo} with  arbitrary real coefficients). Then
\[
||f-\mathcal{F}_m(f)||_2<||f- p||_2.
\]

\item[(ii)] $\lim_{m\to\infty} ||f-\mathcal{F}_m(f)||_{\mathcal{C}^1}=0.$

\item[(iii)]  Plancherel's theorem:
\[
||f||_2^2=\frac{a_0^2}4+\frac 12 \sum_{k=1}^{\infty} \big(a_k^2 + b_k^2\big).
\]

\item[(iv)] A consequence of Plancherel's theorem:
\[
||f-\mathcal{F}_m(f)||_2^2=\frac 12 \sum_{k>m}^{\infty} \big(a_k^2 +
b_k^2\big)\ge \frac 12\big(a_{m+1}^2+b_{m+1}^2\big).
\]

\end{enumerate}

\end{proposition}

Consider the curve $\tilde z_\varepsilon(t)$ given in Theorem
\ref{th1}, which is strictly transversal to the flow \eqref{eq1}. The next result shows that there always exists a trigonometric curve of the form \eqref{curve} of degree $m,$ high enough, and with
coefficients in $\Q,$ which is also strictly transversal   to the
flow \eqref{eq1}. Also we prove that if the Fourier series of a
limit cycle $\gamma$ has a coefficient with a ``high'' value, then
until its corresponding harmonic has been passed (that is, until we
take $m$ higher than the index of this harmonic) one cannot ensure that the
trigonometric curve $W^{(m)}(t)$ constructed in section \ref{sect2}
is near enough to the curve $\tilde z_\varepsilon(t).$ See the definition of the curve $W^{(m)}(t)$ in (\ref{www}).

\begin{theorem}\label{fita}
(i) Let $\gamma(t)=(\gamma_1(t),\gamma_2(t))$ be a $T$-periodic
limit cycle of system \eqref{eq1}. Let $|\varepsilon|>0$ be small
enough, such that the $T$-periodic closed curve given in
Theorem~\ref{eq1}, $\tilde z_\varepsilon(t)$ associated to
$\gamma(t)$, is strictly transversal to the flow given
by~\eqref{eq1}. Then if $m=m(\varepsilon)$ is high enough, there is a
$T$-periodic trigonometric curve of degree $m$ and  rational
coefficients which is also strictly transversal to the flow given
by~\eqref{eq1}.

(ii) Taking $|\varepsilon|$ smaller, if necessary, it holds that
\begin{align*} ||\tilde z_\varepsilon-W^{(m)}||_{\mathcal{C}^1}>
\frac1{\sqrt{2}}||\gamma_j-\mathcal{F}_m(\gamma_j)||_2\ge\frac1{2}\sqrt{a_{m+1}^2+b_{m+1}^2},
\end{align*}
where $j$ is either $1$ or $2$ and  $a_{m+1}$ and $b_{m+1}$ are the coefficients of the $m+1$
harmonics of the Fourier series of  $\gamma_j(t).$
\end{theorem}

\begin{proof} (i) It is clear that if  $\{z(s)\,:\, s\in[0,T]\}$ and $\{\bar z(s)\,:\, s\in[0,T]\}
\}$ are two  $T$-periodic  $\mathcal{C}^1$ closed curves, one of
them is strictly transversal to the flow ~\eqref{eq1} and $|| z-\bar
z||_{\mathcal{C}^1}$ is small enough, then the other curve is
strictly transversal as well. By Proposition~\ref{propiet} (ii) it
holds that for $m$ high enough there exists a $T$-periodic
trigonometric polynomial curve $\tilde W(t)$, of degree $m$ with real
coefficients and such that $||\tilde z_\varepsilon
-\tilde W||_{\mathcal{C}^1}$ is as small as desired. Taking rational
approximations of its coefficients with enough accuracy we get a new curve $W(t)$ that
proves item (i).

(ii) Fix for instance $j=1$. We write
$W^{(m)}=(W_1^{(m)},W_2^{(m)})$ where
 $W_1^{(m)}(t)=  w^{(m)}_1\left( \frac{2\pi}{T}t\right),$ see \eqref{www}.  Then
\begin{align}
||\tilde z_\varepsilon-W^{(m)}||_{\mathcal{C}^1}&>||\tilde
x_\varepsilon-W_1^{(m)}||_{\mathcal{C}^1}> ||\tilde
x_\varepsilon-W_1^{(m)}||_\infty\nonumber\\&\ge ||\tilde
x_\varepsilon-W_1^{(m)}||_2\ge ||\tilde
x_\varepsilon-\mathcal{F}_m(\tilde x_\varepsilon)||_2,\label{desi}
\end{align}
where in the last inequality we have used Proposition~\ref{propiet}
(i), that states that the Fourier polynomial is the best
approximation of a function, considering the norm $L_2$.

Since the curves $\tilde z_\varepsilon$ tend uniformly to $\gamma$
when $\varepsilon$ goes to zero, we have that for $|\varepsilon|$
small enough
\[
||\tilde x_\varepsilon-\mathcal{F}_m(\tilde x_\varepsilon)||_2>\frac1{\sqrt{2}}||\gamma_1-\mathcal{F}_m(\gamma_1)||_2.
\]
In the previous inequality we have chosen the value $1/\sqrt{2}$. We could have chosen any positive value lower than $1$. Since $x_\varepsilon$ tends uniformly to $\gamma_1$ when $\varepsilon$ goes to zero, we have that the quantity $||\tilde x_\varepsilon-\mathcal{F}_m(\tilde x_\varepsilon)||_2$ is close to the quantity $||\gamma_1-\mathcal{F}_m(\gamma_1)||_2$ when $\varepsilon$ tends to zero. For $|\varepsilon|$ small enough, one exceeds the other by a positive constant lower than $1$. If one takes a smaller value of $|\varepsilon|$ this constant can be reduced.

Then, from \eqref{desi},
\[
||\tilde z_\varepsilon-W^{(m)}||_{\mathcal{C}^1}>\frac1{\sqrt{2}}
||\gamma_1-\mathcal{F}_m(\gamma_1)||_2\ge\frac1{2}\sqrt{a_{m+1}^2+b_{m+1}^2},
\]
where we have used Proposition \ref{propiet} (iv). Then the theorem
follows.
\end{proof}

\subsection{Fourier coefficients of systems (\ref{vdpe1}) and
(\ref{Bruss13})}

{F}rom Theorem \ref{fita} we know that for  having a good enough
approximation to the curve  $\tilde z_\varepsilon(t)$ given in
Theorem \ref{th1} by a trigonometric polynomial curve we need to
consider $m$ such that the coefficients of the $m$ harmonics of the
Fourier series of $\gamma(t)$ are small enough.

\begin{table}[ht]
\begin{center}
\begin{tabular}{|c||c|c|c|c|c|c|c|}
\hline
 $m$ &1&3& 5-7& 9&$11-13$&$15-17$&19\\
\hline\hline Coeff. &$1$& $10^{-1}$&
$10^{-2}$& $10^{-3}$ &$10^{-4}$&$10^{-5}$&$10^{-6}$\\
\hline
\end{tabular}
\end{center}
\vspace{0.2cm} \caption{Order of magnitude of the coefficients of
the $m$ harmonics of the Fourier series of the first component of
the limit cycle of the van der Pol system. When $m$ is even all the
coefficients are zero.}\label{table:vdp1}
\end{table}

\begin{table}[ht]
\begin{center}
\begin{tabular}{|c||c|c|c|c|c|c|c|}
\hline
 $m$ &$0-2$&$3-8$& $9-16$& $17-20$\\
\hline\hline Coeff. &$1$& $10^{-1}$&
$10^{-2}$& $10^{-3}$\\
\hline
\end{tabular}
\end{center}
\vspace{0.2cm} \caption{Order of magnitude of the coefficients of
the $m$ harmonics of the Fourier series of the first component of
the limit cycle of the Brusselator system.}\label{table:vdp}
\end{table}

Therefore the number of points $n=2m+1$ used to construct our curve
$W^{(m)}(t)$ is strongly related with the size of the Fourier
coefficients of $\gamma(t)$. These coefficients can be numerically
obtained before starting our process for obtaining a Poincar\'{e}
annular region for proving the existence of a periodic orbit.  In
Tables \ref{table:vdp1} and \ref{table:vdp} we show the order of
magnitude of them for the first component $\gamma_1(t)$ of the limit
cycles  $\gamma(t)$ of the van der Pol \eqref{vdpe1} and the
Brusselator \eqref{Bruss13} systems. The results for the second
component are essentially the same. Notice that the modulus of the
coefficients of the harmonics in the Brusselator system descend much
more slowly than in the van der Pol system, giving a clear
explanation of the harder difficulty for finding trigonometric
curves without contact for the Brusselator system.

\section{The Rychkov system}\label{sect33}

The aim of this section is to prove Theorem~\ref{rych}. As we have
already said in the introduction we consider the system studied by
Rychkov in 1975, see \cite{Rychkov75},
\begin{equation} \label{eqex3} \dot{x}   =   y-\left( x^5-\mu x^3 + \delta x
\right), \quad \dot{y}   =   -x, \end{equation} with $\delta, \mu
\in \mathbb{R}$. The semi-axis $\Sigma   :=   \left\{ (x_0,0) \in
\mathbb{R}^2   :   x_0>0 \right\}$ is a transversal section. This
system is also studied in \cite{Alsholm92, GiaNeu98, Odani96}. The
following features of system (\ref{eqex3}) can be found in the
 aforementioned references. The origin is the only finite singular point and it is a
 focus. Rychkov \cite{Rychkov75} proved that it has at most two limit cycles and that
 for $\delta<0$ there exists a unique limit cycle, which is stable.
 The line $\delta=0$ is a curve of occurrence of Hopf bifurcations.
 When $\mu>0$ there is a curve of bifurcation values $\delta   =   \Delta(\mu)$
 of a saddle-node bifurcation of limit cycles. Odani \cite{Odani96} proved that if $\delta>0$ and
 $0<\delta < \mu^2/5$, then the system has two limit cycles. Figure
 \ref{Rychov_bif} represents the bifurcation diagram of the Rychkov
 system (\ref{eqex3}) in the $(\delta,\mu)$-plane.

\begin{figure}[htb]
\includegraphics[width=0.6\textwidth]{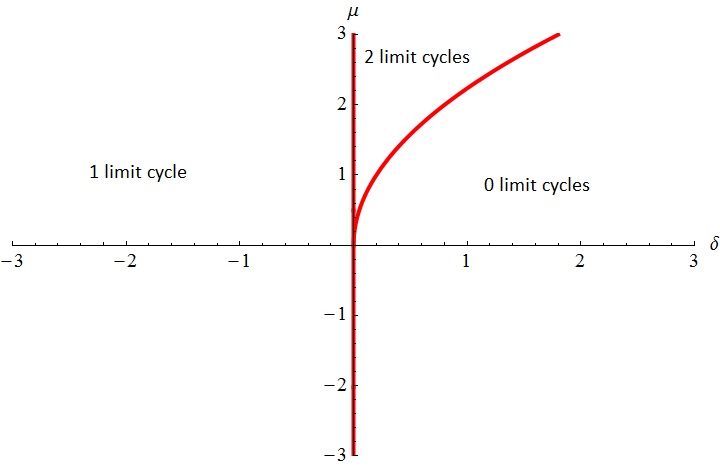}
\caption{Bifurcation diagram of the Rychkov system (\ref{eqex3}).
The curve of the saddle-node bifurcation is qualitative.}
\label{Rychov_bif}
\end{figure}

Here we fix $\mu=1$ and we are interested in finding sharp bounds
for $\delta^*=\Delta (1)$. Since the Rychkov system is a semi-complete family of rotated vector
fields with respect to $\delta$, see~\cite{Duff,Per93,Perko} it
holds that:

\begin{itemize}

\item It for $\delta=\bar\delta$ the system has two limit cycles then $\bar{\delta}<\delta^*.$
Therefore, to prove the inequality $0.224\,{<}\,\delta^*$, it
suffices to prove that the Rychkov system has two limit cycles for
$\delta=0.224.$

\item Similarly, if for $\delta=\hat\delta$ the system has no limit cycle then
$\delta^*<\hat\delta.$ Then, to prove the inequality
$\delta^*<\,\,0.2249654$, it suffices to prove that the Rychkov system
has no limit cycle for $\delta=\,0.2249654.$

\end{itemize}

Therefore, the proof of Theorem \ref{rych} can be reduced to the
study of the above two given values of $\delta$. We study each case
in a different subsection.

\subsection{The proof that system \eqref{rych1} has two limit cycles for
$\delta=0.224$}

Although it suffices to study the case $\delta=0.224$, we prefer to
study also the smaller values of $\delta,$ $0.2$ and $0.22$ to see
how the two limit cycles evolve with the parameter. In the three cases, the origin is
a strong stable focus, the smaller limit cycle is hyperbolic and unstable and the bigger limit cycle is hyperbolic and stable.

\smallbreak

\subsubsection{The case $ \delta=0.2$} The limit cycles cut $\Sigma$ at
$x_0 \sim 0.632018$ and $x_0 \sim
 0.893787$.  By our method we have been able to construct three transversal curves
 which provide two Poincar\'e--Bendixson regions. These regions allow to locate
 each one of the limit cycles.
 The interior transversal curve cuts $\Sigma$ at $x_0 =0.474059$,
 it has been obtained from the unstable limit cycle taking $\varepsilon   =
 0.1$ and $m=5$. The transversal curve in the middle cuts $\Sigma$ at $x_0 =0.711158$,
  it has been obtained from the unstable limit cycle taking $\varepsilon   =
   -0.05$ and $m=7$. The exterior transversal curve cuts $\Sigma$ at $x_0 =1.00597$,
   it has been obtained from the stable limit cycle taking $\varepsilon   =   -0.1$ and $m=5$.
   These curves, together with the limit cycles are represented in Figure \ref{Rychkov-d02}.

\begin{figure}[htb]
\includegraphics[width=0.5\textwidth]{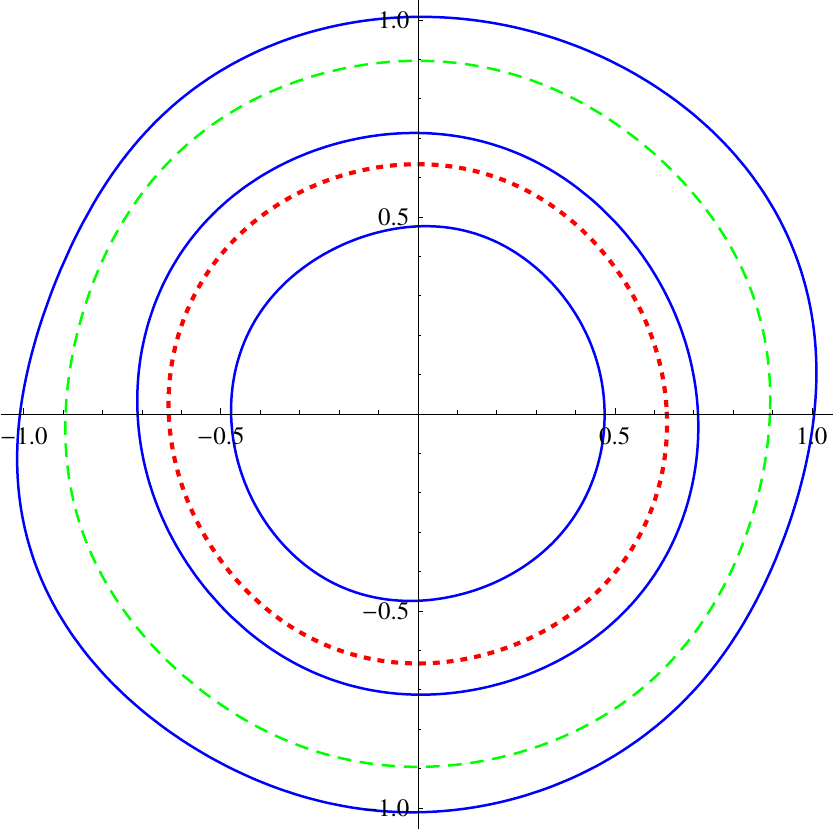}
\caption{Transversal curves are represented in blue and the limit cycles in dotted red (unstable) and
dashed green (stable) for the Rychkov system with $\mu=1$ and $\delta=0.2$.} \label{Rychkov-d02}
\end{figure}

\subsubsection{The case $ \delta=0.22$}
In this case the limit cycles cut $\Sigma$ at $x_0   \sim
0.714276$ and $x_0 \sim 0.830266$. The interior transversal
 curve cuts $\Sigma$ at $x_0 =0.57421$, it has been obtained from the unstable limit
 cycle taking $\varepsilon   =   0.1$ and $m=7$. The transversal curve in the middle
 cuts $\Sigma$ at $x_0 =0.74227$, it has been obtained from the unstable limit cycle
 taking $\varepsilon   =   -0.02$ and $m=7$. The exterior transversal curve
 cuts $\Sigma$ at $x_0 =0.8905$, it has been obtained from the stable limit cycle
 taking $\varepsilon   =   -0.05$ and $m=7$.

\subsubsection{The case $ \delta=0.224$} For this value of $\delta$ the
limit cycles cut $\Sigma$ at $x_0   \sim   0.748705$ and $x_0   \sim
0.799588$. As in the previous case, we have been
 able to find three transversal curves which analytically prove the existence of the
 two limit cycles. The interior transversal curve cuts $\Sigma$ at $x_0 =0.615043$,
 it has been obtained from the unstable limit cycle taking $\varepsilon   =   0.1$
 and $m=7$. The transversal curve in the middle cuts $\Sigma$ at $x_0 =0.75939$,
 it has been obtained from the unstable limit cycle taking $\varepsilon   =   -0.008$
 and $m=10$. The exterior transversal curve cuts $\Sigma$ at $x_0 =0.862111$,
 it has been obtained from the stable limit cycle taking $\varepsilon   =   -0.05$
 and $m=7$. See Figure \ref{Rychkov-d0224}.

\begin{figure}[htb]
\includegraphics[width=0.5\textwidth]{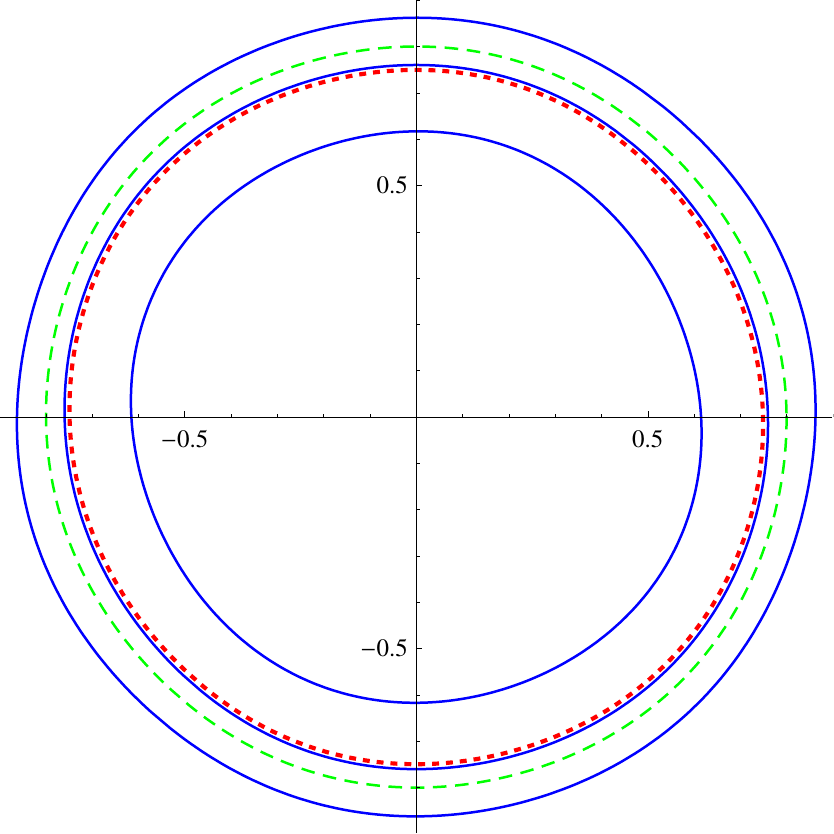}
\caption{Transversal curves are represented in blue and the limit cycles in dotted red (unstable) and
dashed green (stable) for the Rychkov system with $\mu=1$ and $\delta=0.224$.} \label{Rychkov-d0224}
\end{figure}

\subsection{The proof that system \eqref{rych1} has
no limit cycle for $\delta=0.2249654$}

Before proving the second part of the theorem we need some
preliminary results.

The first lemma recalls a classical method for proving non-existence
of periodic orbits. We state and prove it on the plane, but  notice
that it works in any dimension.

\begin{lemma}\label{dot} Let $\mathcal{U}$ be an open subset of $\R^2$ and let
 $B:\mathcal{U}\to\R$ be
a class $\mathcal{C}^1$ function such that its total derivative with
respect to the flow \eqref{eq1},
\[
\dot B(x,y)= \nabla B(x,y)\cdot X(x,y)=\frac{\partial
B(x,y)}{\partial x} P(x,y)+\frac{\partial B(x,y)}{\partial y}
Q(x,y)
\]
does not change sign on $\mathcal{U}$ and vanishes on a set of zero
Lebesgue measure which is not invariant by the flow of \eqref{eq1}.
Then the system \eqref{eq1} has not periodic orbits totally contained in
$\mathcal{U}.$
\end{lemma}

\begin{proof} Let $z(t)=(x(t),y(t))$ be any solution of \eqref{eq1},
contained  in $\mathcal U$ for $t\in [t_1,t_2],$  $t_1<t_2.$  Then,
\[
0\ne\int_{t_1}^{t_2} \dot B(z(t))\,dt= B(z(t_2))-B(z(t_1).
\]
Hence the orbit  cannot be periodic, as we wanted to prove.
\end{proof}

The next result is an adaptation of \cite[Thm. 3]{Cher} to our interests. We sketch its proof.

\begin{proposition}\mbox{\rm (\cite{Cher})} Given a classical polynomial
Li\'{e}nard system
\begin{equation}\label{liena}
\dot x= y-F(x),\quad  \dot y=-x
\end{equation}
and $n\in\N$, there exists a unique polynomial  $B_n(x,y)= \sum_{i=0}^{n} B_i(x)y^i$ such that
$B_n(0,y)=y^n$ and its total derivative with respect to \eqref{liena} is a polynomial that does
not depend on $y.$ \label{prop9}
\end{proposition}

\begin{proof}
We have that
\begin{align*} \dot B_n(x,y)=& \frac{\partial B_n(x,y)}{\partial x}\big(y-F(x)\big)-
x\,\frac{\partial B_n(x,y)}{\partial y}\\
=&\,\,\Big( \sum_{i=0}^{n} B_i'(x)y^i \Big)\big(y-F(x)\big)-
x\,\sum_{i=1}^{n}i\, B_i(x)y^{i-1}\\= &
\,\, B_n'(x)y^{n+1}+  \big(B_{n-1}'(x)-F(x)B'_n(x)\big)y^{n}\\
&+\sum_{k=1}^{n-1}\Big(B_{k-1}'(x)-F(x)B_{k}'(x)-(k+1)xB_{k+1}(x)\Big)y^{k}\\&
-F(x)B_0'(x)-xB_1(x).
\end{align*}
We impose the conditions $B_n(0)=1, B_k(0)=0$ for $k=0,1,\ldots, n-1.$ Then we can solve step by
step the trivial linear differential equations given by the vanishing of the coefficients of
$y^{n+1}, y^n,\ldots$ until $y$. We obtain that $B_n(x)\equiv1,$ $B_{n-1}(x)\equiv0,$
$B_{n-2}(x)=nx^2/2$,
\[
B_{n-3}(x)=n\int_0^x sF(s)\,ds,\quad B_{n-4}(x) =n\int_0^x
sF^2(s)\,ds+\frac{n(n-2)}8 x^4
\]
and so on. Finally $\dot B(x,y)=-F(x)B_0'(x)-xB_1(x),$ as we wanted to prove.
\end{proof}

\noindent{\it The proof that the Rychkov system with $\mu=1$ and $\delta=0.2249654$ has no limit
cycle.} Applying Proposition~\ref{prop9} to system~\eqref{rych1},
\begin{equation*} \dot{x} \, = \, y-\left( x^5- x^3 + \delta x \right), \quad
\dot{y} \, = \, -x,\end{equation*} we get that
\[
\dot B_n(x,y)=x^nR_{4(n-1)}(x,\delta),
\]
where $R_{4(n-1)}$ is an even  polynomial in $x$ of degree $4(n-1).$
For instance, taking $n=4$ we get that
\begin{align*}
B_4(x,y)=&y^4+2 x^2y^2+\frac{4}{105} x^3 \left( 15 x^4-21 x ^2+35
\delta \right) y\\&+\frac{1}{30} x^4 \left( 10 x^8-24 x^6+30 \delta
x^4+15 x^4-40 \delta x^2+30 \delta^2+30 \right)
\end{align*}
and
\begin{align*}
R_{12}(x,\delta )= -\frac4{105}\Big(&105 {x}^{12}-315 {x}^{10}+
\left( 315 \delta +315 \right) {x}^{8}\\&
 -\left( 630 \delta +105 \right) {x}^{6}+ \left( 315 {\delta }^{2}+315 \delta +120
 \right) {x}^{4}\\&- \left( 315 {\delta }^{2}+126 \right) {x}^{2}+35 \delta
 \left( 3 {\delta }^{2}+4 \right)\Big).
\end{align*}

The discriminant of the above polynomial with respect to $x$, except
for some non-zero rational constant factor, is
\begin{align*}
\delta  \big( 3 {\delta }^{2}+4 \big)  \Big( &4233600000 {\delta
}^{7}-4953312000 {\delta }^{6}+59568485760 {\delta }^{5}\\&-
65416468320 {\delta }^{4}+256186378380 {\delta }^{3}-171344748015
{\delta }^{2}\\&+ 250762344740 \delta -52896972996 \Big)^{2}.
\end{align*} By using once more the Sturm's approach we can prove
that it only has one positive zero at $\delta=\delta_4\approx
0.2362516\ldots.$ Therefore, it is not difficult to prove that if
$\delta\ge0.236252$ then $R_{12}(x,\delta)<0.$ In fact, for our
interests it suffices to prove that $R_{12}(x, 0.236252)<0$ for all
$x\in\R.$ From this fact, for $\delta= 0.236252$, we have $\dot
B_{4}(x,y)\le0,$ and it vanishes only at $x=0$. Then, by
Lemma~\ref{dot} we know that for this value of $\delta$ the
system~\eqref{rych1} has no limit cycle. As a consequence we have that
$\delta^*<0.236252.$

Repeating the above procedure for different values of $n$ (even) we
improve the upper bound for $\delta^*$. Our results are presented in
Table~\ref{tab1}.

\begin{table}[ht]
\begin{center}
\begin{tabular}{|c||c|c|c|c|c|c|c|}
\hline
 $n$ &$50$&$100$& $150$& $200$& $250$&$300$\\
\hline\hline $\mbox{Bound}$&$0.2252$& $ 0.2251$&
$0.2250$& $0.2249715$ &$0.2249676$ &$0.2249654$\\
\hline
\end{tabular}
\end{center}
\caption{Upper bounds for $\delta^*$
for the Rychkov system \eqref{rych1} with $\mu \, = \, 1$.}\label{tab1}
\end{table}

It is remarkable  that increasing $n\le 300$ we have found that there exist values $\delta_n$
such that for $\delta>\delta_n$ it holds that $R_{4(n-1)}(x,\delta)<0$ and moreover that these
values seem to decrease monotonically towards $\delta^*.$ Observe also that for the case $n=300$
we must prove that the even polynomial of degree 1196, $R_{1196}(x,0.2249654),$ which has rational
coefficients, has no real roots. \qed

\section*{Acknowledgements}

The first author is partially supported by Spanish Government with
the grant MTM2013-40998-P and by Generalitat de Catalunya Government
with the grant 2014SGR568. The second and third authors are partially
 supported by a MINECO/FEDER grant number MTM2014-53703-P and by an
 AGAUR (Generalitat de Catalunya) grant number 2014SGR1204.

\end{document}